\documentclass[a4paper, 10pt, twoside, notitlepage]{amsart}

\usepackage[utf8]{inputenc}
\usepackage{color}
\usepackage{amsmath} 
\usepackage{amssymb} 
\usepackage{amsthm}
\usepackage{geometry}
\usepackage{graphicx}
\usepackage{mathtools}
\usepackage{esint}
\usepackage[colorlinks=true,linkcolor=blue]{hyperref}

\theoremstyle{plain}
\newtheorem{thm}{Theorem}
\newtheorem{prop}{Proposition}[section]

\newtheorem{cor}[prop]{Corollary}

\newtheorem{defi}[prop]{Definition}
\newtheorem{rmk}[prop]{Remark}

\newcommand {\R} {\mathbb{R}} 
 \newcommand {\N} {\mathbb{N}}
 
\newcommand {\p} {\partial}

\newcommand {\D} {\Delta}

\begin{document}

\title[On Runge approximation for a finite-dimensional Schrödinger inverse problems]{ On Runge approximation and Lipschitz stability for a finite-dimensional Schrödinger inverse problem}

\author{Angkana Rüland}
\address{Max-Planck Institute for Mathematics in the Sciences, Inselstraße 22, 04103 Leipzig, Germany}
\email{rueland@mis.mpg.de}

\author{Eva Sincich}
\address{Dipartimento di Matematica e Geoscienze
Università degli Studi di Trieste,
via Valerio 12/1, 34127 Trieste
Italy}
\email{esincich@units.it}

\maketitle

\centerline{\textit{Dedicated to Sergio Vessella on the  occasion of his
$65^{th}$ birthday}}

\begin{abstract}
In this note we reprove the Lipschitz stability for the inverse problem for the Schrödinger operator with finite-dimensional potentials by using quantitative Runge approximation results. This provides a quantification of the Schrödinger version of the argument from \cite{KV85} and presents a slight variant of the strategy considered in \cite{AdHGS17} which may prove useful also in the context of more general operators.
\end{abstract}

\section{Introduction}
\label{introduction}

In this note we reprove the Lipschitz stability for the inverse problem for the Schrödinger operator
\begin{align}
\label{eq:L}
L_q= -\D + q
\end{align}
 with a piecewise linear potential $q$ satisfying a suitable spectral condition (see \eqref{spectral} and Remark \ref{rmk:Robin_to_Dirichlet} below). This had previously been derived in \cite{AdHGS17} by means of singular solutions and quantitative unique continuation estimates. In our version of the Lipschitz stability proof, we split the stability problem into two clearly separated steps: 
\begin{itemize}
\item[(i)] \emph{A boundary recovery result for which we rely on the argument from \cite{AdHGS17} }. 
This a typical initial step when proving stability for inverse boundary value problem (see for instance \cite{A1,Sy-U-2,Brown}).
\item[(ii)] \emph{A quantitative Runge approximation result for which we rely on a slight variant of the argument from \cite{RS17c} adapted to the present geometry.} This and related quantitative Runge approximation results hold for very general operators (involving for instance variable coefficients and lower order drift terms). Provided that boundary recovery results are available for these operators (which would allow to apply step (i)), it is thus possible to carry out our strategy of deducing Lipschitz estimates also for more general operators.
\end{itemize}

While we hope that the ideas which are used in this article will be useful also in more general settings, for simplicity of presentation, in this note we restrict our attention to the Schrödinger setting with linear potentials. 

Let us put this into a context. As it is well-known that both the Calder\'on problem and the inverse problem for the Schrödinger operator are highly unstable and thus pose major difficulties for instance for numerics, in \cite{AV} Alessandrini and Vessella posed the question of the existence of quantities or settings in which improvements in terms of stability are possible. As discovered in \cite{AV}, considering potentials or conductivities in certain \emph{finite-dimensional} spaces provides such a scenario. In order to deduce this, the argument from \cite{AV} relied on a combination of singular solutions  \cite{A1} (see also the method of localized potentials \cite{G08, HPS17} providing concentrated information as well) and unique continuation estimates (see for instance \cite{ARRV09}). Building on this observation, a tremendous amount of activity has revolved around extending this to more complex equations and systems (see for instance \cite{BF11, AdHGS16, GS15, AdHGS17, BdHQ13, AS18, H19, H19a} for only some of the activities in this area). 

Let us also recall the alternative approach  proposed by Bacchelli and Vessella in \cite{BaV06} for proving a Lipschitz stability estimate for an unknown polygonal boundary. This provides a quite general procedure to deal with stability for nonlinear and finite-dimensional inverse problems.  Indeed, their argument shows that if for a nonlinear inverse problem a global and constructive,
although very weak, stability estimate is available and also one has
local Lipschitz stability, with a known constant and known radius of
validity, then also global Lipschitz stability holds and the Lipschitz
constant can be concretely evaluated. Other papers, as for instance \cite{BdHFV15,BdHFVZ17}, were influenced by this argument. 

Instead of directly combining singular solutions with unique continuation results as previously done in the literature, in this article we present a slight modification of this approach. It will allow us to clearly isolate the two steps (i), (ii) from above. A similar approach had already been introduced in the context of stability estimates for an inverse Schrödinger equation with partial data in which the potential is known in a neighbourhood of the boundary (see \cite{RS17c}) and in a non-local analogue of the problem that is investigated here \cite{RS18}. We hope that the splitting of the problem into the two steps (i), (ii) adds clarity to the structure of our argument and of similar proofs, and that it will be of use in more general settings.

\section{Conditional Lipschitz Stability - Assumptions and the Result}\label{subsection assumptions}
In this section, we present the assumptions under which we will be working in the sequel and state our main result.
\subsection{Notation and definitions}\label{subsec notation and definitions}

In several places within this manuscript it will be useful to single out one coordinate
direction. To this purpose, the following notation for
points $x\in \mathbb{R}^n$ will be adopted. For $n\geq 3$,
a point $x\in \mathbb{R}^n$ will be denoted by
$x=(x',x_n)$, where $x'\in\mathbb{R}^{n-1}$ and $x_n\in\mathbb{R}$.
Moreover, given a point $x\in \mathbb{R}^n$,
we will denote with $B_r(x), B_r'(x)$ the open balls in
$\mathbb{R}^{n},\mathbb{R}^{n-1}$ respectively centered at $x$ with radius $r$
and by $Q_r(x)$ the cylinder

\[Q_r(x)=B_r'(x')\times(x_n-r,x_n+r).\]




\begin{defi}\label{def Lipschitz boundary}
Let $\Omega$ be a domain in $\mathbb R^n$. We say that a portion
$\Sigma$ of $\partial\Omega$ is of Lipschitz class with constants
$r_0,L$ if for any $P\in\Sigma$ there exists a rigid
transformation of $\mathbb R^n$ under which we have $P=0$ and
$$\Omega\cap Q_{r_0}=\{x\in Q_{r_0}\,:\,x_n>\varphi(x')\},$$
where $\varphi$ is a Lipschitz function on $B'_{r_0}$ satisfying

\[\varphi(0)=0;\qquad
\|\varphi\|_{C^{0,1}(B'_{r_0})}\leq Lr_0.\]

In the sequel, it is understood that $\partial\Omega$ is of Lipschitz class with
constants $r_0,L$ as a special case of $\Sigma$, with
$\Sigma=\partial\Omega$.
\end{defi}

\begin{defi}\label{flat portion}
Let $\Omega$ be a domain in $\mathbb R^n$. We say that a portion $\Sigma$ of
$\partial\Omega$ is a flat portion of size $r_0$
if for any $P\in\Sigma$ there exists a rigid transformation of
$\mathbb R^n$ under which we have $P=0$ and


\begin{align}
\begin{split}
\Sigma\cap{Q}_{r_{0}/3} &=\{x\in
Q_{r_0/3}|x_n=0\},\\
\Omega\cap {Q}_{r_{0}/3} &=\{x\in
Q_{r_0/3}|x_n>0\},\\
\left(\mathbb{R}^{n}\setminus\Omega\right)\cap {Q}_{r_{0}/3} &= \{x\in
Q_{r_0/3}|x_n<0\},
\end{split}
\end{align}
\end{defi}



\begin{defi}\label{DN}
Let $\Omega$ be a domain in $\mathbb{R}^n$ with Lipschitz boundary
$\partial\Omega$ and $\Sigma$ an open non-empty (flat) open portion of
$\partial\Omega$. Let us introduce $H^{\frac{1}{2}}_{00}(\Sigma)$ (\cite{LiM}, {Chapter $1$}),  the subspace of
$H^{\frac{1}{2}}(\partial\Omega)$ which is the closure of 
\begin{equation}\label{Hco}
H^{\frac{1}{2}}_{co}(\Sigma)=\big\{f\in
H^{\frac{1}{2}}(\partial\Omega) \:\vert\:\textnormal{supp}
\:f\subset\Sigma\big\}
\end{equation}
in the norm of $H^{\frac{1}{2}}(\partial\Omega)$. The dual  space of $H^{\frac{1}{2}}_{00}(\Sigma)$ will be denoted as ${H}^{-\frac{1}{2}}_{00}(\Sigma)$.

Let $q\in L^{\infty}(\Omega)$ and  assume that $0$ is not an eigenvalue of $(-\Delta + q)$ with Dirichlet boundary conditions in $\Omega$, i.e., 
\begin{eqnarray}
\{u\in H^1_0(\Omega): (-\Delta +q) u=0) \} =\{ 0\},
\end{eqnarray}
then the local Dirichlet-to-Neumann map associated to $q$ and
$\Sigma$ is the operator
\begin{equation}\label{mappaDN}
\Lambda_{q}^{\Sigma}:H^{\frac{1}{2}}_{00}(\Sigma)\longrightarrow
{H}^{-\frac{1}{2}}_{00}(\Sigma),
\end{equation}
 defined by
\begin{equation}\label{def DN locale}
<\Lambda_{q}^{\Sigma}\:g,\:f>\:=\:\int_{\:\Omega}
\nabla u \cdot \nabla  {v} + q u  {v}\:dx,
\end{equation}
for any $g$, $f\in H^{\frac{1}{2}}_{00}(\Sigma)$, where
$u\in{H}^{1}(\Omega)$ is the weak solution to
\begin{displaymath}
\left\{ \begin{array}{ll} (-\Delta + q(x))u=0, &
\textrm{$\textnormal{in}\quad\Omega$},\\
u=g, & \textrm{$\textnormal{on}\quad{\partial\Omega},$}
\end{array} \right.
\end{displaymath}
and $v\in H^{1}(\Omega)$ is any function such that
$v\vert_{\partial\Omega}=f$ in the trace sense. Here we
denote by $<\cdot,\:\cdot>$ the $L^{2}(\partial\Omega)$-pairing
between $H^{\frac{1}{2}}_{00}(\Sigma)$ and its dual
$H^{-\frac{1}{2}}_{00}(\Sigma)$.
\end{defi}



\subsection{Assumptions about the domain $\Omega$}\label{subsec assumption domain}


We assume that $\Omega$ is a domain in $\mathbb{R}^n$
and that there is a positive constant $B$ such that

\begin{equation}\label{assumption Omega}
|\Omega|\leq B r_0 ^n,
\end{equation}
where $|\Omega|$ denotes the Lebesgue measure of $\Omega$.


We fix an open non-empty subset $\Sigma$ of $\partial\Omega$
(where the measurements in terms of  the local Dirichlet-to-Neumann map are taken). We consider
 \[\bar\Omega = \bigcup_{j=1}^{N}\bar{D}_j,\]
where $D_j$, $j=1,\dots , N$ are known open sets of
$\mathbb{R}^n$, satisfying the conditions (1)-(3) below.

\begin{enumerate}
\item $D_j$, $j=1,\dots , N$ are connected and pairwise
nonoverlapping domains.

\item $\partial{D}_j$, $j=1,\dots , N$ are of Lipschitz class with
constants $r_0$, $L$.

\item There exists one region, say $D_1$, such that
$\partial{D}_1\cap\Sigma$ contains a  flat portion
$\Sigma_1$ of size $r_0$ and for every $i\in\{2,\dots , N\}$ there exists $j_1,\dots, 
j_K\in\{1,\dots , N\}$ such that

\begin{equation}\label{catena dominii}
D_{j_1}=D_1,\qquad D_{j_K}=D_i,
\end{equation}
and such that 
\begin{eqnarray}
\mathop {\left( \bigcup_{k=1}^l \overline{D_{j_k}}\right)}\limits^ \circ \ \ \ \mbox{and} \ \ \ \mathop {\left( \Omega \setminus \bigcup_{k=1}^ l \overline{D_{j_k}}\right)}\limits^ \circ \ , \ \ l=1,\dots, K
\end{eqnarray}
are Lipschitz domains. 
In addition we assume that, for every $k=1,\dots , K$,
$\partial{D}_{j_k}\cap \partial{D}_{j_{k-1}}$ contains a flat 
portion $\Sigma_k$ of size $r_0$ (here we agree that
$D_{j_0}=\mathbb{R}^n\setminus\Omega$), such that

\[\Sigma_k\subset\Omega,\quad\mbox{for\:every}\:k=2,\dots , K\]

\end{enumerate}
and for any $k=1,\dots,K$, there exists $P_k\in \Sigma_k$ and a rigid transformation of coordinates  under which we have $P_k=0$ and 
\begin{align*}
\Sigma_k\cap Q_{{r_0}/3} &=\{x\in Q_{{r_0}/3} : x_n=0 \},\\
 D_{j_k}\cap Q_{{r_0}/3} &=\{x\in Q_{{r_0}/3} : x_n>0 \},\\
 D_{j_{k-1}}\cap Q_{{r_0}/3} &=\{x\in Q_{{r_0}/3} : x_n<0 \}. 
\end{align*}

\subsection{A-priori information on the potential $q$}\label{subsec assumption potential}

We shall consider a real valued function $q\in L^{\infty}(\Omega)$, with
\begin{equation}\label{apriori q}
||q||_{L^{\infty}(\Omega)}\leq E_0,
\end{equation}
for some positive constant $E_0$ and of type
\begin{subequations}
\begin{eqnarray}\label{a priori info su q}
& &q(x)=\sum_{j=1}^{N}q^{j}(x)\chi_{D_j}(x),\qquad
x\in\Omega,\label{potential 1a}\\
& & q^{j}(x)=a^j+A^j\cdot x\label{potential 2a},
\end{eqnarray}
\end{subequations}
where $a^j\in\mathbb{R}$, $A^j\in\mathbb{R}^n$ and $D_j$, $j=1,\dots ,
N$ are the given subdomains introduced in section \ref{subsec assumption
domain}.

Let $B$, $N$, $r_0$, $L$, $E_0$ be given
positive numbers with $N\in\mathbb{N}$. We will
refer to this set of numbers, along with the space dimension $n$,
as to the \textit{a-priori data}. Several constants depending on the \textit{a-priori data} will appear within the paper. In order to simplify our notation, any quantity denoted by $C,\tilde{C},c_1,c_2, \dots$ will be called a \emph{constant} understanding in most cases that it only depends on the a-priori data.

Observe that the class of functions of the form \eqref{potential 1a} - \eqref{potential 2a} is a finite-dimensional linear space. The $L^{\infty}$-norm $||q||_{L^{\infty}(\Omega)}$ is equivalent to the norm
\[||| q|||=\textnormal{max}_{j=1,\dots , N}\left\{|a^j|+|A^j|\right\}\]
modulo constants which only depend on the a-priori data.


\subsection{Normalization and spectral assumptions}
In the sequel, without loss of generality, we make a few normalization assumptions and introduce the precise spectral conditions on which our result is based.

First, let $K\in \{ 1,\dots,N \}$ be such that 

\begin{eqnarray}
E=\|q_1-q_2\|_{L^{\infty}(\Omega)}=\|q_1-q_2\|_{L^{\infty}(D_K)} ,
\end{eqnarray}
and recall that there exists $j_1,\dots, j_K \in \{1,\dots N\}$ such that 
\begin{eqnarray}
D_{j_1}=D_1, \dots, D_{j_K}=D_K \ .
\end{eqnarray}
With no loss of generality, we may rearrange the indices of these subdomains so that the above mentioned chain is simply denoted by $D_1,\dots, D_K, K\le N$. We also introduce the following sets
\begin{eqnarray}
{\mathcal{W}}_k= \mathop {\left( \bigcup_{i=1}^k\overline{D_{i}}\right)}\limits^ \circ  \ , \ \ \ {\mathcal{U}}_k=\Omega \setminus \overline{{\mathcal{W}}_k}. 
\end{eqnarray}

We require further spectral conditions which could however be relaxed at the expense of passing from the Dirichlet-to-Neumann to the complex Robin-to-Dirichlet map (see Remark \ref{rmk:Robin_to_Dirichlet}). More precisely, we assume that
 $0$ is not an eigenvalue of $(-\Delta + q)$ with Dirichlet boundary conditions in ${\mathcal{U}}_k$, i.e., 
\begin{eqnarray}\label{spectral}
\{u\in H^1_0({\mathcal{U}}_k): (-\Delta +q) u=0 \} =\{ 0\} \  \ \mbox{for any} \ \ k=0, \dots, K .
\end{eqnarray}
Moreover, on the geometry of the sets $\Sigma_k$ we additionally require that
\begin{align}
\label{eq:boundary_restrict}
\begin{split}
\Sigma_k &\subset \partial \mathcal{U}_{k-1} \setminus \overline{\partial \mathcal{U}_{k}},\ 
\Sigma_{k+1} \subset \partial \mathcal{U}_{k} \setminus \overline{\partial \mathcal{U}_{k-1}}.
\end{split}
\end{align}

Analogously as in \eqref{def DN locale}, for each $k\in\{1,\dots,N\}$ we introduce the following Dirichlet-to-Neumann maps 

\begin{eqnarray}
\Lambda_{q_i}^{\Sigma_{k+1}} : H_{00}^{\frac{1}{2}}(\Sigma_{k+1}) \rightarrow H_{00}^{-\frac{1}{2}}(\Sigma_{k+1})
\end{eqnarray}
for the domain ${\mathcal{U}}_k$ relative to potentials $q_i$ and localized on $\Sigma_{k+1}$, for $i=1,2$.

In addition we set ${\mathcal{U}}_0=\Omega$ and $\Lambda_{q_i}^{\Sigma_{1}}=\Lambda_{q_i}^{\Sigma}$.

\subsection{The main result}

Under the explained conditions, we reprove the Lipschitz stability of the finite-dimensional Schrödinger inverse problem:

\begin{thm}
\label{thm:Lip_Runge}
Let $\Omega \subset \R^n$ with $n\geq 3$,  $D_1,\dots,D_N$ and $\Sigma$ satisfy the assumptions from above.
Let $q_1,q_2 \in L^{\infty}(\Omega)$ be two potentials satisfying \eqref{apriori q} and the spectral conditions \eqref{spectral} (below) for $q=q_i$ for $i=1,2$. Moreover, let  $q_1,q_2 \in L^{\infty}(\Omega)$ be of the type
 \begin{subequations}
\begin{eqnarray}\label{a priori info su q}
& &q_i(x)=\sum_{j=1}^{N}{q_i}^{j}(x)\chi_{D_j}(x),\qquad
x\in\Omega,\label{potential 1}\\
& & {q_i}^{j}(x)=a^j+A^j\cdot x\label{potential 2},
\end{eqnarray}
\end{subequations}
where $a^j\in\mathbb{R}$, $A^j\in\mathbb{R}^n$, then there exists a constant $C_N>0$ depending on the a-priori data only  such that
\begin{align*}
\|q_1-q_2\|_{L^{\infty}(\Omega)} \leq C_N \|\Lambda_{q_1}^{\Sigma} - \Lambda_{q_2}^{\Sigma}\|_{*}.
\end{align*} 
Moreover, the dependence on $N$ for the constant $C_N>0$ can be explicitly estimated: $C_N\leq f\circ f \circ ... \circ f(C)$, where $f(t) = \exp(C_0t^{\mu})$ for some constants $C_0>0$, $C>0$, $\mu>0$ which are  independent of $N$. The concatenation of the functions $f$ can be at most $N$-fold.
\end{thm}

\vskip 0.2cm

\begin{rmk}
\label{rmk:spec}
Let us comment on the spectral conditions. Since by the variational characterisation for Lipschitz domains the eigenvalues of the operator $-\D +q$ depend continuously on the domain (see for instance \cite{BV65, F99}), and by the eigenvalue monotonicity, we obtain that zero as a Dirichlet eigenvalue is not stable under domain perturbation. In this sense, generically, zero is not a Dirichlet eigenvalue for a given operator and (a slight variation) of a domain. Hence, while imposing this condition here, generically, this is not the case anyway.
\end{rmk}

\begin{rmk}
\label{rmk:Robin_to_Dirichlet}
In contrast to \cite{AdHGS17}, as already observed, we have imposed spectral conditions. 
However, we remark that with an argument along the same lines as the present one, it would also have been possible to deal with the Robin-to-Dirichlet map with local complex Robin condition instead. 
This has the advantage that in this case \emph{no} spectral conditions have to be imposed as the underlying Robin boundary value problem is well posed \cite{BamDu87} (see also \cite[Section 3]{AdHGS17} in which the associated Green's function with local complex Robin condition is constructed and estimated).
\end{rmk}

\begin{rmk}
\label{rmk:functions} 
Although we here focus on the piecewise affine functions defined in \eqref{potential 1}, \eqref{potential 2}, we observe that our strategy can be applied to a much larger class of potentials. In fact, any space of linearly independent functions $\psi_1,\dots, \psi_m$ such that also the restrictions
\begin{align}
\label{eq:lin_ind}
\psi_1|_{\Sigma_k},\dots, \psi_m|_{\Sigma_k}
\end{align} 
are linearly independent can be chosen. Further variations of this are possible: If boundary recovery results for the normal derivative are available, it for instance suffices to prove that for some choice of $\gamma_j \in \{0,1\}$ the functions
\begin{align}
\label{eq:lin_ind_a}
\p_{\nu}^{\gamma_1}\psi_1|_{\Sigma_k},\dots, \p_{\nu}^{\gamma_m}\psi_m|_{\Sigma_k}
\end{align} 
are linearly independent. This allows us to recover the coefficients of $\psi(x) = \sum\limits_{j=1}^{m}a_m \psi_m(x)$ from the boundary measurements. Analogous remarks hold, if higher normal derivatives can be recovered on the boundary.
\end{rmk}

\section{Two Ingredients: Stability at the Boundary and Runge Approximation}\label{stability at the boundary}

In the following we collect our two main ingredients in the Lipschitz stability proof. Here we discuss the boundary stability result and the quantitative Runge approximation.

\subsection{Stability at the boundary}
We recall the first of our two main ingredients, namely the stability of the potential $q$ on the boundary which is contained in the proof of Theorem 2.2 in  \cite{AdHGS17}. Contrary to  the inverse conductivity problem for which it is well-known that the stability at the boundary for the conductivity coefficient is of Lipschitz type \cite{Sy-U-2, A1}, for  the potential $q$ the stability is of  H\"older type. 

\begin{thm}[Theorem 2.2  in \cite{AdHGS17}]
\label{prop:boundary_det}
Let $L_{q_1}, L_{q_2}$ be the Schrödinger operators from above with potentials $q_1,q_2\in L^{\infty}(\Omega)$. Then there exist constants $C>1$, $\eta\in (0,1)$, depending on the a-priori data only, such that
\begin{align*}
&\|q_1-q_2\|_{L^{\infty}(\Sigma_k \cap B_{r_0/4}(P_k) )} +|\partial_{\nu}(q_1-q_2)(P_k) | \\
& \leq C
(\|\Lambda_{q_1}^{\Sigma_k}-\Lambda_{q_2}^{\Sigma_k}\|_{\ast} + \|q_1-q_2\|_{L^{\infty}(\Omega)})^{1-\eta}{\|\Lambda_{q_1}^{\Sigma_k}-\Lambda_{q_2}^{\Sigma_k}\|^{\eta}_{\ast} }.
\end{align*}
\end{thm}

Moreover, by using our a-priori assumption on the finite-dimensional feature of the potentials from Theorem \ref{prop:boundary_det}, we can immediately deduce the following estimate in the interior. 

\begin{cor}
\label{prop:boundary_det_finite}
Let the hypotheses of Theorem \ref {thm:Lip_Runge} be satisfied. 
Then there exist constants $C>1$, $\eta\in (0,1)$, depending on the a-priori data only, such that
\begin{align}
\label{eq:boundary_det}
\|q_1-q_2\|_{L^{\infty}(D_k)}  \leq C
(\|\Lambda_{q_1}^{\Sigma_k}-\Lambda_{q_2}^{\Sigma_k}\|_{\ast} + \|q_1-q_2\|_{L^{\infty}(\Omega)})^{1-\eta}{\|\Lambda_{q_1}^{\Sigma_k}-\Lambda_{q_2}^{\Sigma_k}\|^{\eta}_{\ast} }.
\end{align}

\end{cor}

\begin{proof}[Proof of Corollary \ref{prop:boundary_det_finite}]
We recall an argument introduced in \cite{AdHGS16}. Let $x\in \overline{D_k}$, let us define
\begin{equation}\label{gamma linear on Dk}
\alpha_k+\beta_k\cdot x = (q_1 -q_2)(x),
\end{equation}
and let us denote by $\{e_j\}_{j=1,\dots, n-1}$ a family of $n-1$ orthonormal vectors, defining the hyperplane containing the flat part of $\Sigma_k$. By computing $q_1 -q_2$ on the points $P_k$, $P_k+\frac{r_0}{5}e_j$, $j=1,\dots, n-1$, taking their differences and applying the following estimate (coming from Theorem \ref {prop:boundary_det})
\begin{align*}
\|q_1-q_2\|_{L^{\infty}(\Sigma_k \cap B_{r_0/4}(P_k) )}  \leq C
(\|\Lambda_{q_1}^{\Sigma_k}-\Lambda_{q_2}^{\Sigma_k}\|_{\ast} + \|q_1-q_2\|_{L^{\infty}(\Omega)})^{1-\eta}{\|\Lambda_{q_1}^{\Sigma_k}-\Lambda_{q_2}^{\Sigma_k}\|^{\eta}_{\ast} },
\end{align*}
we obtain

\begin{eqnarray}
|\alpha_k+\beta_k\cdot P_k| &\leq & C
(\|\Lambda_{q_1}^{\Sigma_k}-\Lambda_{q_2}^{\Sigma_k}\|_{\ast} + \|q_1-q_2\|_{L^{\infty}(\Omega)})^{1-\eta}{\|\Lambda_{q_1}^{\Sigma_k}-\Lambda_{q_2}^{\Sigma_k}\|^{\eta}_{\ast} },\ \\
|\beta_k\cdot e_j| &\leq & C
(\|\Lambda_{q_1}^{\Sigma_k}-\Lambda_{q_2}^{\Sigma_k}\|_{\ast} + \|q_1-q_2\|_{L^{\infty}(\Omega)})^{1-\eta}{\|\Lambda_{q_1}^{\Sigma_k}-\Lambda_{q_2}^{\Sigma_k}\|^{\eta}_{\ast}, }\label{estimates beta along the directions ej}
\end{eqnarray}
for $j=1,\dots, n-1$, where $C>0$ is  a constant depending on the a-priori data only. To estimate $\beta_k$ along the remaining direction $\nu$ and therefore $\alpha_k$, we use the following bound on the normal derivative (coming from Theorem \ref{prop:boundary_det})
\begin{align*}
|\partial_{\nu}(q_1-q_2)(P_k) | \leq C
(\|\Lambda_{q_1}^{\Sigma_k}-\Lambda_{q_2}^{\Sigma_k}\|_{\ast} + \|q_1-q_2\|_{L^{\infty}(\Omega)})^{1-\eta}{\|\Lambda_{q_1}^{\Sigma_k}-\Lambda_{q_2}^{\Sigma_k}\|^{\eta}_{\ast} } \ .
\end{align*}

\end{proof}

\subsection{Runge approximation}\label{Runge approximation}
In order to propagate the information on the potential along the jump interfaces of the potential, we use (a slight variant of) the quantitative Runge approximation result from \cite{RS17c} which constitutes our second main ingredient. Let us recall the notation from \cite{RS17c} for that: we consider domains $\Omega_1, \Omega_2 \subset \R^n$ which are bounded, open and Lipschitz. Further we assume that $\Omega_1 \subset \Omega_2$ such that 
\begin{itemize}
\item[(R)] $\Omega_2 \setminus \overline{\Omega_1}$ is connected with $\overline{\Gamma} \subset \partial \Omega_2 \setminus \overline{\partial \Omega_1}$ being relatively open, non-empty and Lipschitz regular.
\end{itemize}
We remark that in contrast to the set-up in \cite{RS17c} we do not require that $\Omega_1 \Subset \Omega_2$ but allow for $\Omega_1$ and $\Omega_2$ to share a part of its boundary if the restriction of $u$ to the boundary vanishes on this part of the boundary.
Denoting by $L$ the operator from \eqref{eq:L}, we further define the sets  
\begin{align*}
S_1 &:= \{u\in L^2(\Omega_1): \ Lu= 0 \mbox{ in } \Omega_1 \},\\
S_2 &:= \{u\in L^2(\Omega_2): \ Lu = 0 \mbox{ in } \Omega_2 , \ u|_{\partial \Omega_2} \in H^{\frac{1}{2}}_{00}(\Gamma)\}.
\end{align*}
With this notation in place, we recall below the quantitative Runge approximation result from \cite{RS17c}.

\begin{thm}[Theorem 2 in \cite{RS17c}]
\label{thm:Runge_quant}
Let $L$ be the operator from \eqref{eq:L} and let $\Omega_1, \Omega_2,\Gamma$ and  $S_1,S_2$ be as above and assume that both $\Omega_1$ and $\Omega_2$ are domains such that $0$ is not a Dirichlet eigenvalue of $L$.
There exist a parameter $\mu>0$ and a constant $C>1$ (depending on $\Omega_1, \Omega_2, \Gamma, n, \|q\|_{L^{\infty}(\Omega_2)}$) such that for each function $h \in S_1$ and each error threshold $\epsilon\in(0,1)$, there exists a function $u \in S_2$ with
\begin{align}
\label{eq:approx}
\|h-u|_{\Omega_1}\|_{L^2(\Omega_1)} \leq \epsilon \|h\|_{H^{1}(\Omega_1)}, \ \ \ \ 
\|u\|_{H^{1/2}_{00}(\Gamma)} \leq C e^{C \epsilon^{-\mu}} \|h\|_{L^2(\Omega_1)}.
\end{align}
\end{thm}

Due to the explained (slight) modification of the statement of Theorem \ref{thm:Runge_quant} with respect to the result from \cite{RS17c} we briefly explain the argument showing that essentially no change in the proof of \cite{RS17c} is necessary.

\begin{proof}[Proof of Theorem \ref{thm:Runge_quant}]
As in \cite{RS17c} the proof follows in two steps. As a first step, we consider the dual equation
\begin{align*}
(-\D + q) w & = h \chi_{\Omega_1} \mbox{ in } \Omega_2,\\
w & = 0 \mbox{ on } \partial \Omega_2,
\end{align*}
where $\chi_{\Omega_1}$ denotes the characteristic function of $\Omega_1$. Then by the same arguments as in \cite{RS17c} we obtain the quantitative unique continuation result
\begin{align*}
\|w \|_{H^1(\Omega_2 \setminus \overline{\Omega_1})}
\leq C \frac{\|h\|_{L^2(\Omega_1)}}{\left(\log\left(C \frac{\|h\|_{L^2(\Omega_1)}}{\|\p_{\nu} w\|_{H^{-\frac{1}{2}}_{00}(\Gamma)}} \right)\right)^{\nu}}.
\end{align*}
As the second step, as in \cite{RS17c} we note that for the set $X:= \overline{S_1}^{L^2(\Omega_1)}$ (which is a Hilbert space), the operator 
\begin{align*}
A: H^{\frac{1}{2}}_{00}(\Gamma) \rightarrow X \subset L^2(\Omega_1), \ g \mapsto u|_{\Omega_1},
\end{align*}
where $u$ is a solution to $Lu = 0$ in $\Omega_2$ such that $u=g$ on $\partial \Omega_2$ in the trace sense, is a compact, injective operator with dense range. The spectral theorem thus yields bases $\{\varphi_j\}_{j\in \N} \subset H^{\frac{1}{2}}_{00}(\Gamma)$ and $\{\psi_j\}_{j\in \N} \subset X$ and singular values $\{\sigma_j\}_{j\in \N}$ such that $A \varphi_j = \sigma_j \psi_j$ and $A^{\ast}\psi_j = \sigma_j \varphi_j$. Given these, we argue by the abstract, quantitative duality argument in exactly the same way as in \cite{RS17c} (where in the estimate for $\|A(R_{\alpha} h) - h \|_{L^2(\Omega_1)}$ we exploit that $w_{\alpha}|_{\partial \Omega_1 \cap \partial \Omega_2}=0$, so that 
$
\|w_{\alpha}\|_{H^{\frac{1}{2}}(\partial \Omega_1)} = \|w_{\alpha}\|_{H^{\frac{1}{2}}(\partial \Omega_1 \setminus \partial \Omega_1)} \leq C \|w_{\alpha}\|_{H^1(\Omega_2 \setminus \overline{\Omega_1})}$). This implies the desired result.
\end{proof}

Using the a-priori information on the potential $q$, we will iterate the boundary recovery result by the aid of the quantitative Runge approximation property and Alessandrini's identity.

\section{Proof of Theorem \ref{thm:Lip_Runge}}

We apply the Runge approximation result from Theorem \ref{thm:Runge_quant} combined with Alessandrini's identity.  Hence let ${\varphi}_i \in H^{\frac{1}{2}}_{00}(\Sigma_{k+1}) ,\  i=1,2$ and consider $u_1,u_2 \in  H^{1}({{\mathcal{U}}_k})$ solutions to 
\begin{eqnarray}\label{Uk}
\left\{ \begin{array}{ll} (-\Delta + q_i(x))u_i=0, &
\textrm{$\textnormal{in}\quad {{\mathcal{U}}_k}$},\\
u_i={\varphi}_i, & \textrm{$\textnormal{on}\quad{\partial {{\mathcal{U}}_k}}$.}
\end{array} \right.
\end{eqnarray}

Then, by Theorem \ref{thm:Runge_quant} there exist solutions $v_1, v_2 \in H^{1}({{\mathcal{U}}_{k-1}})$ of $L_i v_i = 0$ in ${\mathcal{U}}_{k-1}$ with Dirichlet traces $v_i|_{\partial {{\mathcal{U}}_{k-1}} } \in H^{\frac{1}{2}}_{00}(\Sigma_{k})$ and
\begin{align*}
\|v_i-u_i\|_{L^2({\mathcal{U}}_k)} \leq \epsilon \|u_i\|_{H^1({\mathcal{U}}_k)}, \ \|v_i\|_{H^{1/2}_{00}(\Sigma_k)}\leq C e^{C \epsilon^{-\mu}}\|u_i\|_{L^2({\mathcal{U}}_k)}.
\end{align*}
Then, Alessandrini's identity yields the following control:
\begin{align}
\label{eq:Aless}
((\Lambda_{q_1}^{\Sigma_k}-\Lambda_{q_2}^{\Sigma_k})v_1,v_2)
= \int\limits_{{\mathcal{U}}_{k-1}} (q_1-q_2)v_1 v_2 dx
= \int\limits_{{\mathcal{U}}_{k}} (q_1-q_2)v_1 v_2 dx 
+ \int\limits_{D_k} (q_1-q_2)v_1 v_2 dx .
\end{align}
For the first term on the right hand side of \eqref{eq:Aless}, we rewrite:
\begin{align}
\label{eq:complement}
\begin{split}
\int\limits_{{\mathcal{U}}_{k}} (q_1-q_2)v_1 v_2 dx 
&= \int\limits_{{\mathcal{U}}_{k}} (q_1-q_2)(v_1-u_1) v_2 dx 
+ \int\limits_{{\mathcal{U}}_{k}} (q_1-q_2)u_1 (v_2-u_2) dx\\
& \quad + \int\limits_{{\mathcal{U}}_{k}} (q_1-q_2)u_1 u_2 dx\\
&= \int\limits_{{\mathcal{U}}_{k}} (q_1-q_2)(v_1-u_1) v_2 dx 
+ \int\limits_{{\mathcal{U}}_{k}} (q_1-q_2)u_1 (v_2-u_2) dx\\
& \quad + ((\Lambda_{q_1}^{\Sigma_{k+1}}-\Lambda_{q_2}^{\Sigma_{k+1}})u_1,u_2).
\end{split}
\end{align}
Next we combine the two identities \eqref{eq:Aless} and \eqref{eq:complement} obtaining
\begin{align*}
&|((\Lambda_{q_1}^{\Sigma_{k+1}}-\Lambda_{q_2}^{\Sigma_{k+1}})\varphi_1,\varphi_2)|\\
&\leq |((\Lambda_{q_1}^{\Sigma_k}-\Lambda_{q_2}^{\Sigma_k})v_1,v_2)|\\
& \quad + \|q_1-q_2\|_{L^{\infty}({\mathcal{U}}_{k})}(\|u_1-v_1\|_{L^2({\mathcal{U}}_{k})}\|v_2\|_{L^2({\mathcal{U}}_{k})} + \|u_1\|_{L^2({\mathcal{U}}_{k})}\|u_2-v_2\|_{L^2({\mathcal{U}}_{k})} )\\
& \quad  + \|q_1-q_2\|_{L^{\infty}(D_k)} \|v_1\|_{L^2(D_k)} \|v_2\|_{L^2(D_k)}.
\end{align*}
Morover the above inequality, the estimate
\begin{align*}
\|v_2\|_{L^2({\mathcal{U}}_{k})} \leq \|u_2\|_{L^2({\mathcal{U}}_{k})} + \|u_2-v_2\|_{L^2({\mathcal{U}}_{k})},
\end{align*}
and standard bounds for elliptic equations lead to 
\begin{align*}
&|((\Lambda_{q_1}^{\Sigma_{k+1}}-\Lambda_{q_2}^{\Sigma_{k+1}})\varphi_1,\varphi_2)| \\
 &\leq  \|\Lambda_{q_1}^{\Sigma_k}-\Lambda_{q_2}^{\Sigma_k}\|_{\ast}\|v_1\|_{H_{00}^{1/2}(\Sigma_k)}\|v_2\|_{H_{00}^{1/2}(\Sigma_k)} \\
&\quad  + \|q_1-q_2\|_{L^{\infty}({\mathcal{U}}_{k})}(\|u_1-v_1\|_{L^2({\mathcal{U}}_{k})}\|u_2\|_{L^2({\mathcal{U}}_{k})} \\
& \qquad  +\|u_1-v_1\|_{L^2({\mathcal{U}}_{k})}\|u_2-v_2\|_{L^2({\mathcal{U}}_{k})}+ \|u_1\|_{L^2({\mathcal{U}}_{k})}\|u_2-v_2\|_{L^2({\mathcal{U}}_{k})} )\\
& \quad  + \|q_1-q_2\|_{L^{\infty}(D_k)} \|v_1\|_{L^2(D_k)} \|v_2\|_{L^2(D_k)}\\
&\leq \|\Lambda_{q_1}^{\Sigma_k}-\Lambda_{q_2}^{\Sigma_k}\|_{\ast}\|v_1\|_{H_{00}^{1/2}(\Sigma_k)}\|v_2\|_{H_{00}^{1/2}(\Sigma_k)} \\
& \quad +C\|q_1-q_2\|_{L^{\infty}({\mathcal{U}}_{k})}(\|u_1-v_1\|_{L^2({\mathcal{U}}_{k})}\|u_2\|_{H_{00}^{1/2}({\Sigma}_{k+1})} )\\
& \quad +C\|q_1-q_2\|_{L^{\infty}({\mathcal{U}}_{k})}(\|u_1-v_1\|_{L^2({\mathcal{U}}_{k})}\|u_2-v_2\|_{L^2({\mathcal{U}}_{k})}+ \|u_1\|_{H_{00}^{1/2}({\Sigma}_{k+1})}\|u_2-v_2\|_{L^2({\mathcal{U}}_{k})} )\\
& \quad +C \|q_1-q_2\|_{L^{\infty}(D_k)} \|v_1\|_{H_{00}^{1/2}(\Sigma_k)} \|v_2\|_{H_{00}^{1/2}(\Sigma_k)}.
\end{align*}
Invoking the Runge approximation property and Corollary \ref{prop:boundary_det_finite} then allows us to further bound 
\begin{align*}
|((\Lambda_{q_1}^{\Sigma_{k+1}}-\Lambda_{q_2}^{\Sigma_{k+1}})\varphi_1,\varphi_2)|
&\leq \|\Lambda_{q_1}^{\Sigma_k}-\Lambda_{q_2}^{\Sigma_k}\|_{\ast}\|v_1\|_{H_{00}^{1/2}(\Sigma_k)}\|v_2\|_{H_{00}^{1/2}(\Sigma_k)}\\
& \quad +C \epsilon \|q_1-q_2\|_{L^{\infty}{\mathcal{U}}_{k})} \|u_1\|_{H_{00}^{1/2}(\Sigma_{k+1})} \|u_2\|_{H_{00}^{1/2}({\Sigma}_{k+1})}\\
& \quad  + C e^{C \epsilon^{-\mu}}(\|\Lambda_{q_1}^{\Sigma_k}-\Lambda_{q_2}^{\Sigma_k}\|_{\ast} + E)\left( \frac{\|\Lambda_{q_1}^{\Sigma_k}-\Lambda_{q_2}^{\Sigma_k}\|_{\ast} }{\|\Lambda_{q_1}^{\Sigma_k}-\Lambda_{q_2}^{\Sigma_k}\|_{\ast} 
+E }\right)^{\eta}\times\\
& \quad \times \|u_1\|_{H_{00}^{1/2}(\Sigma_{k+1})} \|u_2\|_{H_{00}^{1/2}(\Sigma_{k+1})}\\
&\leq C e^{C \epsilon^{-\mu}}\|\Lambda_{q_1}^{\Sigma_k}-\Lambda_{q_2}^{\Sigma_k}\|_{\ast}\|u_1\|_{L^2(\mathcal{U}_{k})}\|u_2\|_{L^2(\mathcal{U}_k)}\\
& \quad +C \epsilon \|q_1-q_2\|_{L^{\infty}({\mathcal{U}}_{k})} \|u_1\|_{H_{00}^{1/2}(\Sigma_{k+1})} \|u_2\|_{H_{00}^{1/2}({\Sigma}_{k+1})}\\
& \quad  + C e^{C \epsilon^{-\mu}}(\|\Lambda_{q_1}^{\Sigma_k}-\Lambda_{q_2}^{\Sigma_k}\|_{\ast} + E)\left( \frac{\|\Lambda_{q_1}^{\Sigma_k}-\Lambda_{q_2}^{\Sigma_k}\|_{\ast} }{\|\Lambda_{q_1}^{\Sigma_k}-\Lambda_{q_2}^{\Sigma_k}\|_{\ast} 
+E }\right)^{\eta}\times\\
& \quad \times \|u_1\|_{H_{00}^{1/2}(\Sigma_{k+1})} \|u_2\|_{H_{00}^{1/2}(\Sigma_{k+1})}.
\end{align*}
Recalling the boundary conditions for the functions $u_j$ from \eqref{Uk}, using energy estimates and estimating $\|q_1 - q_2\|_{L^{\infty}(\mathcal{U}_k)} \leq E$, we further arrive at
\begin{align*}
|((\Lambda_{q_1}^{\Sigma_{k+1}}-\Lambda_{q_2}^{\Sigma_{k+1}})\varphi_1,\varphi_2)|
&\leq  C \epsilon E (\|\varphi_1\|_{H_{00}^{1/2}(\Sigma_{k+1})} \|\varphi_2\|_{H_{00}^{1/2}(\Sigma_{k+1})})\\
& \quad  +\tilde{C} e^{C \epsilon^{-\mu}}(\|\Lambda_{q_1}^{\Sigma_k}-\Lambda_{q_2}^{\Sigma_k}\|_{\ast} + E)\left( \frac{\|\Lambda_{q_1}^{\Sigma}-\Lambda_{q_2}^{\Sigma}\|_{\ast} }{\|\Lambda_{q_1}^{\Sigma}-\Lambda_{q_2}^{\Sigma}\|_{\ast} 
+ E }\right)^{\eta}\times\\
& \quad \times (\|\varphi_1\|_{H_{00}^{1/2}(\Sigma_{k+1})} \|\varphi_2\|_{H_{00}^{1/2}(\Sigma_{k+1})}).
\end{align*}
Taking the sup over all $\varphi_1, \varphi_2 \in H^{\frac{1}{2}}_{00}(\Sigma_{k+1})$ then implies that

\begin{eqnarray}\label{Runge+Bdry}
\|\Lambda_{q_1}^{\Sigma_{k+1}}-\Lambda_{q_2}^{\Sigma_{k+1}}\|_*\le  C e^{C \epsilon^{-\mu}}\left(\|\Lambda_{q_1}^{\Sigma_{k}}-\Lambda_{q_2}^{\Sigma_{k}}\|_* +E\right)\left(\frac{\left(\|\Lambda_{q_1}^{\Sigma_{k}}-\Lambda_{q_2}^{\Sigma_{k}}\|_* \right)}{\|\Lambda_{q_1}^{\Sigma_{k}}-\Lambda_{q_2}^{\Sigma_{k}}\|_* +E}\right)^{\eta} +C\epsilon E .
\end{eqnarray}

\begin{rmk}
\label{rmk:stabmap}
Similarly as in \cite[Remark 2]{RS17c} we observe that, under the hypothesis of Theorem \ref{thm:Lip_Runge},  the following logarithmic dependence of the local Dirichlet to Neumann map over $\Sigma_{k+1}$ from the one over $\Sigma_k$ 
\begin{eqnarray}
\|\Lambda_{q_1}^{\Sigma_{k+1}}-\Lambda_{q_2}^{\Sigma_{k+1}}\|_*\le C |\log(\|\Lambda_{q_1}^{\Sigma_{k}}-\Lambda_{q_2}^{\Sigma_{k}}\|^{-\frac{\eta}{2}}_* )|^{-\frac{1}{\mu}}.
\end{eqnarray}
can be derived. The proof relies on an optimization argument over $\epsilon$ for the right hand side of \eqref{Runge+Bdry}.
See also \cite{AKi12} where the authors provide a quite general method to obtain a continuous dependence of a \emph{global} Dirichlet to Neumann map from a \emph{local} one on a larger domain. 
\end{rmk} 


We introduce the notation 
\begin{eqnarray}
\delta_j &= \|\Lambda_{q_1}^{\Sigma_j}-\Lambda_{q_2}^{\Sigma_j}\|_{\ast}
\end{eqnarray}
for $j\in \{1, \dots, K\}$. 
We start from the last domain of the chain, namely $D_K$, where the maximum is achieved. 
By Corollary \ref{prop:boundary_det_finite} we have 
\begin{eqnarray}\label{K}
E=\|q_1-q_2\|_{L^{{\infty}}(D_K)}\le C (\delta_K + E)\left(\frac{\delta_K}{\delta_K+E} \right)^{\eta} \ . 
\end{eqnarray}
Now we distinguish two cases: 
\begin{description}
\item [Ka)]  $E < \delta_K$,
\item [Kb)]  $E\ge \delta_K$.
\end{description}
If case [Kb)] occurs, then we notice that \eqref{K} leads to 
\begin{eqnarray}\label{K1}
E\le 2CE\left(\frac{\delta_K}{\delta_K+E} \right)^{\eta} \ .
\end{eqnarray}
which in turns gives 
\begin{eqnarray}\label{K2}
\left(\frac{1}{2C}\right)^{\frac{1}{\eta}}\le \left(\frac{\delta_K}{\delta_K+E} \right)\ .
\end{eqnarray}
From the latter, by rearranging, we deduce that 
\begin{eqnarray}\label{K3}
E\le \left({2C}\right)^{\frac{1}{\eta}} \delta_K \ .
\end{eqnarray}
Now, by the estimates \eqref{Runge+Bdry} and \eqref{K3} we have that 
\begin{align*}
E \le C_K C e^{C \epsilon^{-\mu}}(\delta_{K-1} +E)\left( \frac{\delta_{K-1}}{\delta_{K-1}+E} \right)^{\eta} + C_KCE\epsilon, 
\end{align*}
where $C_K=\left({2C}\right)^{\frac{1}{\eta}} $. Choosing $\epsilon= \frac{1}{2 C_K C} $, we can absorb the last term in the above inequality 
\begin{align}
\label{eq:stepK}
E \le c_K(\delta_{K-1} +E)\left( \frac{\delta_{K-1}}{\delta_{K-1}+E} \right)^{\eta}, 
\end{align} 
where $c_K=2C_K C e^{C [{(2C_KC)}]^{\mu}}$.
If the case [Ka)] occurs, we directly obtain \eqref{K3} and \eqref{eq:stepK}.
Dealing with the estimate \eqref{eq:stepK} as above we may again  distinguish two cases 
\begin{description}
\item [K-1a)]  $E < \delta_{K-1}$,
\item [K-1b)]  $E\ge \delta_{K-1}$.
\end{description}
Arguing analogously as above, we conclude that 
\begin{eqnarray}\label{K4}
E\le \left({2c_K}\right)^{\frac{1}{\eta}} \delta_{K-1} \ .
\end{eqnarray}
Now from estimate \eqref{K4} and from \eqref{Runge+Bdry} we in turn observe that
\begin{eqnarray}
E \le C_{K-1} C e^{C \epsilon^{-\mu}}(\delta_{K-2} +E)\left( \frac{\delta_{K-2}}{\delta_{K-2}+E} \right)^{\eta} + C_{K-1}CE\epsilon, 
\end{eqnarray}
where $C_{K-1}=\left({2c_K}\right)^{\frac{1}{\eta}} $.
Choosing $\epsilon= \frac{1}{2 C_{K-1} C} $ we can absorb the last term in the above inequality 
\begin{eqnarray}
E \le c_{K-1}(\delta_{K-2} +E)\left( \frac{\delta_{K-2}}{\delta_{K-1}+E} \right)^{\eta}, 
\end{eqnarray} 
where $c_{K-1}=2C_{K-1} C e^{C [{(2C_{K-1}C)}]^{\mu}}$.
Iterating such an argument we end up with the  estimate 
\begin{eqnarray}
E \le c_{2}(\delta_{1} +E)\left( \frac{\delta_{1}}{\delta_{1}+E} \right)^{\eta}, 
\end{eqnarray} 
which, by the argument above, leads to 
\begin{eqnarray}
E\le (2c_2)^{\frac{1}{\eta}}\delta_1. 
\end{eqnarray}
This concludes the proof. 


\vskip 0.5cm 

{\bf{Acknowledgement}}.  The work of  ES was performed under the PRIN grant No. 201758MTR2-007.  ES has also been supported by the Individual Funding for Basic Research (FFABR) granted by MIUR.


\begin{thebibliography}{AdHGS17}
\bibitem[AlSa18]{AS18}
Giovanni~S. Alberti and Matteo Santacesaria.
\newblock Calder{\'o}n's inverse problem with a finite number of measurements.
\newblock {\em  Forum Math. Sigma}, 7, 2018.


\bibitem[A90]{A1} G. Alessandrini, \newblock Singular Solutions of Elliptic Equations and the Determination of Conductivity by Boundary Measurements, \newblock {\em J. Differential Equations}, \textbf{84}, (2) , 252-272, 1990.


\bibitem[AdHGS16]{AdHGS16}
Giovanni Alessandrini, Maarten~V. de~Hoop, Romina Gaburro, and Eva Sincich.
\newblock Lipschitz stability for the electrostatic inverse boundary value
  problem with piecewise linear conductivities.
\newblock {\em Journal de Math{\'e}matiques Pures et Appliqu{\'e}es},
  107:638--664, 2016.

\bibitem[AdHGS17]{AdHGS17}
Giovanni Alessandrini, Maarten~V. de~Hoop, Romina Gaburro, and Eva Sincich.
\newblock Lipschitz stability for a piecewise linear {S}chr{\"o}dinger
  potential from local {C}auchy data.
\newblock {\em Asymptotic Analysis},  108, 115-149, 2018.

\bibitem[AKi12] {AKi12} Giovanni Alessandrini and Kyoungsun Kim, 
\newblock Single-logarithmic stability for the Calder\'on problem with local data
\newblock {\em  Journal of Inverse and Ill-posed Problems}, 20, (2012) , doi: 10.1515/jip-2012-0014 .

\bibitem[ARdRsV09]{ARRV09}
Giovanni Alessandrini, Luca Rondi, Edi Rosset and Sergio Vessella.
\newblock The stability for the {C}auchy problem for elliptic equations.
\newblock {\em Inverse Problems}, 25(12):123004, 47, 2009.


\bibitem[AV05]{AV}
Giovanni Alessandrini and Sergio Vessella.
\newblock Lipschitz stability for the inverse conductivity problem.
\newblock {\em Advances in Applied Mathematics}, 35(2):207--241, 2005.
\bibitem[BaV06]{BaV06} Valeria Bacchelli and Sergio Vessella.
\newblock Lipschitz stability for a stationary 2D inverse problem with unknown polygonal boundary,
\newblock{\em Inverse Problems}, 22 (5), 2006 

\bibitem[BabVy65]{BV65}
Ivo Babu{\v{s}}ka and Rudolf V{\`y}born{\`y}.
\newblock Continuous dependence of eigenvalues on the domain.
\newblock {\em Czechoslovak Mathematical Journal}, 15(2):169--178, 1965.

\bibitem[BamDu87]{BamDu87}
Axel Bamberger and Ha Tap Duong, 
\newblock Diffraction d'une onde acoustique par une paroi absorbante: Nouvelles equations integrales
\newblock {\em Math. Meth. in the Appl. Sci.}, 9 :431--454, 1987.

\bibitem[BdHQ13]{BdHQ13}
Elena Beretta, Maarten~V. de~Hoop and Lingyun Qiu.
\newblock Lipschitz stability of an inverse boundary value problem for a
  {S}chr{\"o}dinger-type equation.
\newblock {\em SIAM Journal on Mathematical Analysis}, 45(2):679--699, 2013.

\bibitem[BdHFV15]{BdHFV15} 
Elena Beretta, Maarten~V. de~Hoop, Elisa  Francini and Sergio  Vessella.
\newblock Stable determination of polyhedral interfaces from boundary data for the Helmholtz equation.
\newblock {\em Communications in Partial Differential Equations} 40 (7), 1365-1392, 2015 


\bibitem[BdHFVZ17]{BdHFVZ17}
 Elena Beretta, Maarten~V. de~Hoop, Elisa Francini, Sergio Vessella and Jian Zhai.  
\newblock Uniqueness and Lipschitz stability of an inverse boundary value problem for time-harmonic elastic waves.
 \newblock {\em Inverse Problems}, 33(3), 2017
 
\bibitem[BF11]{BF11}
Elena Beretta and Elisa Francini.
\newblock Lipschitz stability for the electrical impedance tomography problem:
  the complex case.
\newblock {\em Comm. Partial Differential Equations}, 36:1723 -- 1749, 2011.


\bibitem[Bro01]{Brown}
Russell~M Brown.
\newblock Recovering the conductivity at the boundary from the {D}irichlet to
  {N}eumann map: a pointwise result.
\newblock {\em Journal of Inverse and Ill-posed Problems}, 9(6):567--574, 2001.

\bibitem[Fu99]{F99}
Bent Fuglede.
\newblock Continuous domain dependence of the eigenvalues of the {D}irichlet
  {L}aplacian and related operators in {H}ilbert space.
\newblock {\em Journal of functional analysis}, 167(1):183--200, 1999.


\bibitem[Geb08]{G08}
Bastian Gebauer.
\newblock Localized potentials in electrical impedance tomography.
\newblock {\em Inverse Probl. Imaging}, 2(2):251--269, 2008.

\bibitem[GS15]{GS15}
Romina Gaburro and Eva Sincich.
\newblock Lipschitz stability for the inverse conductivity problem for a
  conformal class of anisotropic conductivities.
\newblock {\em Inverse Problems}, 31(1):015008, 2015.


\bibitem[Har19a]{H19}
Bastian Harrach.
\newblock Uniqueness and {L}ipschitz stability in electrical impedance
  tomography with finitely many electrodes.
\newblock {\em Inverse Problems}, 35(2):024005, 2019.

\bibitem[Har19b]{H19a}
Bastian Harrach.
\newblock Uniqueness, stability and global convergence for a discrete inverse
  elliptic {R}obin transmission problem.
\newblock {\em arXiv preprint arXiv:1907.02759}, 2019.

\bibitem[HPS17]{HPS17}
Bastian Harrach, Valter Pohjola and Mikko Salo.
\newblock Monotonicity and local uniqueness for the {H}elmholtz equation.
\newblock {\em arXiv preprint arXiv:1709.08756}, 2017.

\bibitem[KVo85]{KV85}
R.~V. Kohn and M.~Vogelius.
\newblock Determining conductivity by boundary measurements. {II}. {I}nterior
  results.
\newblock {\em Comm. Pure Appl. Math.}, 38(5):643--667, 1985.

\bibitem[LiM]{LiM} J.L.Lions and E. Magenes, Non-homogeneous Boundary Value Problems and Applications \textbf{1}, Die Grundlehren der mathematischen Wissenschaften [Fundamental Principles of Mathematical Sciences], \textbf{181}, Springer-Verlag, New York, 1972. Translated from French by P. Kenneth, MR0350177 (50 $\#$2670).

\bibitem[RSal18]{RS17c}
Angkana R{\"u}land and Mikko Salo.
\newblock Quantitative {R}unge approximation and inverse problems.
\newblock {\em International Mathematics Research Notices},
  2019(20):6216--6234, 2018.

\bibitem[RS19]{RS18}
Angkana R{\"u}land and Eva Sincich.
\newblock Lipschitz stability for the finite dimensional fractional
  {C}alder{\'o}n problem with finite {C}auchy data.
\newblock {\em Inverse Problems
  and Imaging}, 13, 1023-1044 201.


\bibitem[SyU87]-{Sy-U-1} J. Sylvester and G. Uhlmann, \newblock A Global Uniqueness Theorem for an Inverse Boundary Valued Problem, \newblock {Ann. of Math.}, \textbf{125} (1987), 153-169.

\bibitem[SyU88]{Sy-U-2}J. Sylvester and G. Uhlmann, \newblock Inverse boundary value problems at the boundary - continuous dependence, \newblock {Comm. Pure Appl. Math.} \textbf{41} (2) (1988), 197-219.

\end{thebibliography}
\end{document}